\documentclass [12 pt]{amsart}
\usepackage{amssymb,latexsym}
\usepackage{xypic}
\theoremstyle{plain}
\newtheorem{Thm}{Theorem}[section]
\newtheorem{Cor}{Corollary}[section]
\newtheorem{Lem}{Lemma}[section]

\newtheorem{Ex}{Example}[section]
\numberwithin{equation}{section}
\theoremstyle{definition}
\newtheorem{Def}{Definition}[section]

\theoremstyle{remark}
\newtheorem{Rem}{\em Remark}

\begin{document}

\title[A NOTE ON TWO CLASSES OF HYPERIDEALS]{A NOTE ON TWO CLASSES OF HYPERIDEALS}

\author{M. Anbarloei}

\address[]{Department of Mathematics, Faculty of Sciences, Imam Khomeini International University, Qazvin, Iran.
}
\email[]{m.anbarloei@sci.ikiu.ac.ir}

\keywords{Prime hyperideal, r-hyperideal, n-hyperideal, Hyperring.}

\subjclass[2000]{20N20}

\thanks{}

\begin{abstract}
Let $R$ be a commutative multiplicative hyperring with 1. In this paper, we define 
the concepts of r-hyperideal and n-hyperideal of the hyperring $R$ which are two new classes of
hyperideals. Several properties of them are provided. A hyperideal $I$ of a multiplicative hyperring $R$ is called an r-hyperideal of $R$, if for all $x,y \in R$, $xoy \subseteq I$ and $ann(x)=\{0\} $, then $y \in I$. Also, a hyperideal $I$ of a multiplicative hyperring $R$ is called an n-hyperideal of $R$, if for all $x,y \in R$, $xoy \subseteq I$ and $x \notin r(0)$, then $y \in I$.
\end{abstract}

\maketitle
\section{Introduction} 
Throughout this paper
our ring $R$ is a commutative multiplicative hyperring with 1. Hyperstructures, as a natural generalization of ordinary algebraic structures, were first initiated by Marty in 1934
\cite{sorc1}, when he defined the hypergroups and began to investigate their properties
with applications to groups, algebraic functions and rational fractions. Later on,
many researchers have observed that the theory of hyperstructures also have many
applications in both applied and pure sciences which a comprehensive review of this
theory can be found in \cite {sorc2} and \cite {sorc4}.
In \cite {sorc4}, Corsini and Leoreanu-Fotea have collected numerous applications of algebraic
hyperstructures, especially those from the last fifteen years to the following subjects: geometry, lattices, binary relations, rough sets and fuzzy sets, cryptography, automata, median algebras, codes, artificial intelligence, relation algebras,
and probabilities.

Motivated from the concept of the prime hyperideals, in this paper, we introduce and study two new classes of hyperideals called r-hyperideals and n-hyperideals. Let $R$ be a commutative multiplicative hyperring. Several properties of them are provided. A hyperideal $I$ of a multiplicative hyperring $R$ is called an r-hyperideal of $R$, if for all $x,y \in R$, $xoy \subseteq I$ and $ann(x)=\{0\} $, then $y \in I$. Also, We define a proper hyperideal $I$ of $R$ as an n-hyperideal if for all $x,y \in R$, $xoy \subseteq I$ and $x \notin r(0)$, then
$y \in I$. \\
The paper is orgnized as follows. In Section 2, we have given some basic definitions and results of multiplicative hyperrings which we need to develop our paper. 
In Section 3, we introduce the concept of r-hyperideals and give some basic properties of them. For example we show (Theorem \ref{8827}) that if $R$ is a reduced hyperring and $P$ is a minimal hyperideal of $R$ such that $s \in R$ is an idempotent element, then $P + ann(s)$ is
an r-hyperideal. Section 4 continues the study of the properties of r-hyperideals. In particular, we discuss the relations between prime hyperideals and r-hyperideals. In Section 5, we introduce the notion of n-hyperideals and give some basic properties of them. For example, we show (Theorem \ref{25}) that if $I$ is a proper hyperideal of $R$, then the following statements are equivalent:\\
(1) $I$ is an n-hyperideal of $R$.\\
(2) $I = (I : a)$ for every $a \notin r(0)$.\\
(3) $I_1oI_2 \subseteq I$ such that $I_1 \cap (R - r(0)) \neq \varnothing$, for hyperideals $I_1, I_2$ of $R$, implies $I_2 \subseteq I$. \\ It is shown (Theorem \ref{oooo}) that if $S$ is an n-multiplicatively closed subset of $R$ and $K$ is an hyperideal of $R$ disjoint from $S$, then there exists a hyperideal $I$ which is maximal in the set of all hyperideals of $R$ disjoint from $S$, containing $K$. Any such hyperideal $I$ is an n-hyperideal of $R$. In the final section, we investigate the stability of n-hyperideals in various hyperring-theoric constructions.

\section{Preliminaries}
Recall first the basic terms and definitions from the hyperring theory. 
A hyperoperation on a non-empty set $G$ is a map $"o":G \times G \longrightarrow P^*(G)$, where $P^*(G)$ is the set of all the nonempty subsets of $G$. An algebraic system $(G,o)$ is called a hypergroupoid. A hypergroupoid $(G,o)$ is called a hypergroup if it satisfies the
following:

(1) $a o (b o c)=(a o b ) o c$, for all $a,b,c \in G$.

(2) $a o G=G o a=G$, for all $a \in G$.

A hypergroupoid with the associative hyperoperation is called a semihypergroup.\\
The notion of multiplicative hyperring as one of the important classes of hyperrings was introduced by Rota \cite{rota}.

A non-empty set $R$ with a operation + and a hyperoperation $o$ is called a {\it multiplicative hyperring} if it satisfies the
following:

(1)$ (R,+)$ is an abelian group;

(2)$ (R,o)$ is a semihypergroup;

(3) for all $a, b, c \in R$, we have $ao(b+c) \subseteq aob+aoc$ and $(b+c)oa \subseteq boa+coa$;

(4) for all $a, b \in R$, we have $ao(-b) = (-a)ob = -(aob)$.\\

For any two nonempty subsets $A$ and $B$ of $R$ and $x \in R$, we define
\[ AoB=\bigcup_{a \in A,\ b \in B}aob, \ \ \ \ Aox=Ao\{x\}.\]

Let R be a commutative multiplicative hyperring. Then

(i) An element $e \in R$ is said to be {\it identity} if $a\in aoe$ for all $a \in R$. 

(ii) An element $e \in R$ is said to be {\it scalar identity} if $a=aoe$ for all $a \in R$.

\begin{Def}
A non-empty subset $I$ of a multiplicative hyperring $R$ is a {\it hyperideal} 

(1) If $a, b \in I$, then $a - b \in I$;

(2) If $x \in I $ and $r \in R$, then $rox \subseteq I$.
\end{Def}
\begin{Def} \cite{davvaz1}
Let $(R_1, +_1, o_1)$ and $(R_2, +_2, o_2)$ be two multiplicative hyperrings. A mapping from
$R_1$ into $R_2$ is said to be a {\it good homomorphism} if for all $x,y \in R_1$, $\phi(x +_1 y) =\phi(x)+_2 \phi(y)$ and $\phi(xo_1y) = \phi(x)o_2 \phi(y)$.
\end{Def}
\begin{Def} \cite{ref1}
A nonzero proper hyperideal $P$ of $R$ is called a {\it prime hyperideal} if $xoy \subseteq P$ for $x,y \in R$ implies that $x \in P$ or $y \in P$.

The intersection of all prime hyperideals of $R$ containing $I$ is called the prime radical of $I$, being denoted by $r(I)$. Therefore, $r(0)$ is the prime radical of the
zero hyperideal in the hyperring $R$. If the multiplicative hyperring $R$ does not have any prime hyperideal containing $I$, we define $r(I)=R$. 
\end{Def}
Let {\bf C} be the class of all finite products of elements of $R$, i.e., ${\bf C} = \{r_1 o r_2o . . . or_n \ : \ r_i \in R, n \in \mathbb{N}\} \subseteq P^{\ast }(R)$. A hyperideal $I$ of $R$ is said to be a {\bf C}-hyperideal of $R$ if, for any $A \in {\bf C}, A \cap I \neq \varnothing $ implies $A \subseteq I$.
Let I be a hyperideal of $R$. Let $D = \{r \in R: r^n \subseteq I \ for \ some \ n \in \mathbb{N}\}$. Then $D \subseteq r(I)$, the prime radical of $I$. The equality holds when $I$ is a {\bf C}-hyperideal of $R$(\cite {ref1}, proposition 3.2). In this paper, we assume that all hyperideals are {\bf C}-hyperideals.
\begin{Def} \cite{ref1}
A nonzero proper hyperideal $Q$ of $R$ is called a {\it primary hyperideal} if $xoy \subseteq Q$ for $x,y \in R$ implies that $x \in Q$ or $y \in r(Q)$. 

Since $r(Q)=P$ is a prime hyperideal of $R$ by Propodition 3.6 in \cite{ref1}, $Q$ is referred to as a P-primary hyperideal of $R$.
\end{Def}
\begin{Def} \cite{amer2}
Let $R$ be a multiplicative hyperring. The $n \times n$ matrix with entries in $R$ is called a hypermatix and the set of all hypermatices of $R$ is denoted by $M_n(R)$.\\
Now, for all $A = (A_{ij})_{n \times n}, B = (B_{ij})_{n \times n} \in P^* (M_n(R)); A \subseteq B$ means $A_{ij} \subseteq B_{ij}$. 
\end{Def}
\begin{Def}
\cite{amer2} Let $R$ be a commutative multiplicative hyperring and $e$ be an identity of $R$. An element $x \in R$ is called {\it invertible}, if there exists $y \in R$, such that $e \in xoy$. So, $R$ is called an {\it invertible} if every element of $R$ is invertible.
\end{Def}
\begin{Def} \cite{amer}
An element $x \in R$ is said to be a {\it zero divizor}, if there exists $0 \neq y \in R$ such that $\{0\}=xoy$. The set of all zero divizor elements in $R$ is denoted by $Z(R)$. 
\end{Def}
\begin{Def}
A hyperring $R$ is called an {\it integral hyperdomain}, if for all $x, y \in R$,
$0 \in x o y$ implies that $x = 0$ or $y = 0$.
\end{Def}
\begin{Def}
A hyperring $R$ is said to be a {\it reduced hyperring} if it has no
nilpotent elements. That is, if $x^n = 0$ for $x \in R$ and a natural number $n$, then $x = 0$.
\end{Def}
\begin{Def} \cite{amer3}
Let $R$ be a multiplicative hyperring. An element $r \in R$ is called {\it regular} if there exists $x \in R$
such that $r \in r^2ox$. $R$ is called {\it regular multiplicative hyperring}, if all the elements in $R$ are regular elements. The set of all regular elements in $R$ is denoted by $V(R)$. 
\end{Def}
\begin{Def}
Let $I,J$ be two hyperideals of $R$ and $x \in R$. Then define:
\[(I:J)=\{r \in R : roJ \subseteq I \}\]
\[ann(x) = \{y \in R : x o y=\{0\}\}\]
\end{Def}
\section{$r$-hyperideals } 
In this section, we introduce the notion of r-hyperideals in a commutative multiplicative hyperring and we
investigate many properties of them.
\begin{Def}
A hyperideal $I$ of a multiplicative hyperring $R$ is called an r-hyperideal of $R$, if for all $x,y \in R$, $xoy \subseteq I$ and $ann(x)=\{0\} $, then $y \in I$.
\end{Def}

\begin{Lem} \label{8821}
Let $I$ be a hyperideal of $R$. Then \\
1) $I$ is an r -hyperideal if and only if $I_1oI_2 \subseteq I$ for some hyperideals $I_1,I_2$ of $R$ such that $I_1 \cap V(R) \neq \varnothing$ implies that $I_2 \subseteq I$ .\\
2) If either $IoI_1 = IoI_2$ or $I \cap I_1 = I \cap I_2$ for some r-hyperideals $I_1,I_2$ of $R$ such that $I \cap V(R) \neq \varnothing$, then $I_1=I_2$ .\\
3) If $IoJ$ is an r -hyperideal of $R$ for some hyperideal $J$ of $R$ such that $J \cap V(R) \neq \varnothing$, then
$I = IoJ$ and so $I$ is an r -hyperideal. 
\end{Lem}
\begin{proof}
It is clear.
\end{proof}
If $S = V(R)$, then we denote
the hyperring $S^{-1}R$ by $T(R)$. Let $\phi: R \longrightarrow T(R)$
be the natural homomorphism. Put $\phi^{-1}(J ) = J ^c$, for each hyperideal $J$ in $T(R)$. For more details about notion of contraction, see \cite{amer2}.
\begin{Thm} \label{8822}
Let R be a commutative
multiplicative hyperring and let $I$ be a hyperideal of $R$. Then the following are equivalent:\\
1) $I$ is an r-hyperideal.\\
2) $<a> \cap I = aoI$ , for all $a \in V(R)$.\\
3) $I = (I : a)$, for all $a \in V(R) \backslash I$.\\
4) $I = J^ c$, for some hyperideal $J$ in $ T(R)$.
\end{Thm}
\begin{proof}
$(1) \Longrightarrow (2)$ Let $I$ is an r-hyperideal. It is clear that $aoI \subseteq <a> \cap I$. Let $x \in <a> \cap I$. There exists some $s \in R$ such that $x \in aos$. Since $a \in V(R)$, we have for some $y \in R$, $a \in a^2oy$. Hence $x \in aos \subseteq a^2oyos=ao(aosoy) \subseteq aoI$, since $I \cap V(R) = \varnothing$.\\
$(2) \Longrightarrow (3)$ It is clear that $I \subseteq (I : a)$. Now let $x \in (I : a)$. Then $aox \subseteq I$ and so $aox \subseteq <a> \cap I$. By part (2), we have $aox \subseteq aoI$. Thus $x \in I$.\\
$(3) \Longrightarrow (4)$ and $(4) \Longrightarrow (1)$ follow from Theorem 4.12 in \cite{amer2}. 
\end{proof} 
\begin{Thm} 
The intersection of a nonempty set of r-hyperideals in $R$, is an r-hyperideal.
\end{Thm}
\begin{proof}
Let $\{I_i\}_{i \in \Omega}$ be a nonempty set of r-hyperideals in $R$ and $I=\bigcap_{i \in \Omega} I_i$. Clearly $I$ is a hyperideal. Now, assume that for $x,y \in R$, $xoy \subseteq I$ such that $ann(x)=\{0\}$. Then we have $xoy \subseteq I_i$ for every $i \in \Omega$. Since $I_i$ is an r-hyperideal for every $i \in \Omega$ and $ann(x)=\{0\}$, we obtain $y \in I_i$. It means $y \in I$.
\end{proof}

\begin{Thm} \label{88024}
If $I$ is an r-hyperideal, then $I \subseteq Z(R)$.
\end{Thm}
\begin{proof}
Let $I \nsubseteq Z(R)$. Then there exists $x \in I$ such that $x \notin Z(R)$. It means that $ann(x)=\{0\}$. Since $I$ is an r-hyperideal and $xo1 \subseteq I$ with $ann(x)=\{0\}$, we conclude that $1 \in I$ which is a contradiction. Thus $I \subseteq Z(R)$.
\end{proof}
\begin{Lem} \label{8824}
For $0 \neq x \in R$, $ann(x)$ is an r-hyperideal of $R$.
\end{Lem}
\begin{proof}
Let $aob \subseteq ann(x)$ and $ann(a) =\{0\}$ for some $0 \neq x \in R$. Thus $aobox=\{0\}$. Since $ann(a) =\{0\}$, we have $box=\{0\}$. It means $b \in ann(x)$. Hence $ann(x)$ is an r-hyperideal of $R$.
\end{proof}

\begin{Thm} \label{8825}
Let $R$ is a commutative multiplicative hyperring with 1. Then the following are equivalent:\\ 
1) $R$ is an integral hyperdomain.\\
2) $\{0\}$ is the only r-hyperideal of $R$. \\
3) For every $x,y \in R$, $ann(xoy) = ann(x) \cup ann(y)$. 
\end{Thm}
\begin{proof}
$(1) \Longrightarrow (2)$ Assume that $I$ is a nonzero proper r-hyperideal of $R$. Then there exists $0 \neq x \in I $. Since $R$ is an integral hyperdomain, then $ann(x)=\{0\}$. Since $xo1 \subseteq I$ and $ann(x)=\{0\}$, we conclude that $1 \in I$ which is a contradiction. Hence $\{0\}$ is the only r-hyperideal of $R$.\\
$(2) \Longrightarrow (3)$ follows from Lemma \ref{8824}.\\
$(3) \Longrightarrow (1)$ Let $xoy =\{0\}$ , for some $x,y \in R$. By the hypothesis, we have $1 \in ann(x) \cup ann(y)$. Hence $x=0$ or $y =0$. Thus $R$ is an integral hyperdomain.
\end{proof}

\begin{Thm} \label{8826}
Let $x+y=1$ for some $x,y \in R$. Then $ann(x) + ann(y)$ is an r -hyperideal. 
\end{Thm}
\begin{proof}
Let $aob \subseteq ann(x) + ann(y)$ and $ann(a) = \{0\}$. It is easy to see that $aoboxoy = \{0\}$. Therefore we have $xoyob = \{0\}$, Since $ann(a) = \{0\}$. Hence $box \subseteq ann(y)$
and $boy \subseteq Ann(x)$. Therefore, $b=bo1= bo(x+y) \subseteq box+ boy\subseteq ann(y) + ann(x)$. Consequently, $ ann(x) + ann(y) $ is an r-hyperideal of $R$. 
\end{proof}

\begin{Thm} \label{8827}
Let $R$ be a reduced hyperring and $P$ be a minimal hyperideal of $R$. Suppose that $s \in R$ is an idempotent element. Then $P + ann(s)$ is
an r-hyperideal. 
\end{Thm}
\begin{proof}
Let $aob \subseteq P + ann(s)$ and $ann(a) = \{0\}$. Therefore for each element $t \in aob$, there exists $\alpha \in P$ and $\beta \in ann(s)$ such that $t=\alpha+\beta$. It is clear that there exists $x \notin P$ with $xos=\{0\}$. Therefore we have $xosoaob =\{0\}$. Since $ann(a) = \{0\}$, we get $xosob =\{0\}\subseteq P$. Hence $sob \subseteq P$. Since $sob+1ob-sob \subseteq P + ann(s)$ and $b=1ob$, we have $b \in P + ann(s)$. Consequently, $ P+ ann(s) $ is an r-hyperideal of $R$. 
\end{proof}
\section{the relations between $r$-hyperideals and prime hyperideals}
In this section, we continue the study of basic properties of r-hyperideals begun in the previous section.
\begin{Thm} \label{8831}
Every maximal r-hyperideal of $R$ is a prime hyperideal. 
\end{Thm}
\begin{proof}
Let $M$ be a maximal r-hyperideal of $R$. Let $aob \subseteq M$ and $a \notin M$. It is easy to see that $(M:a)$ is an r-hyperideal of $R$. Moreover, we have $M \subseteq (M:a)$ and $b \in (M:a)$. Since $M$ is a maximal r-hyperideal of $R$, we get $(M:a)=M$. Thus $b \in M$.
\end{proof}
\begin{Thm}
Let $P$ be a prime hyperideal of $R$. $P$ is an r-hyperideal of $R$ if and only if it consists entirely of zerodivisors .
\end{Thm}
\begin{proof}
$\Longrightarrow$  is clear by Theorem \ref{88024}\\
$\Longleftarrow$
Let $P$ consist entirely of zerodivisors. Let $xoy \subseteq P$ for some $x,y \in R$ and $ann(x)=\{0\}$. Since $P$ is a prime hyperideal of $R$ and $P$ consists entirely of zerodivisors, we get $y \in P$. Consequently, $P$ is an r-hyperideal of $R$.
\end{proof}
\begin{Thm}
Let $R$ be a commutative multiplicative hyperring with 1 and let $P_1, \cdot \cdot \cdot, P_n$ be incomparable prime hyperideals of $R$. If $P=\bigcap_{i=1}^n P_i$
is an r-hyperideal of $R$, then
$P_i$ is an r-hyperideal for all $1 \leq i \leq n$.
\end{Thm}
\begin{proof}
Let $xoy \subseteq P_t$ for some $1 \leq t \leq n$ and $ann(x) = \{0\}$. We must show $y \in P_t$. Suppose that $z \in (\Pi _{k \neq t} P_k) \backslash P_t$. Then we have $xoyoz \subseteq P$. By the assumption, we have $yoz \subseteq P$. It means $ yoz \subseteq P_t$. Since $P_t$ is a prime hyperideal of $R$ and $z \notin P_t$, we conclude that $y \in P_t$.
\end{proof}
Recall that a nonzero hyperideal $I$ of $R$ is called {\it essential} if $I \cap J \neq \{0\}$ for every nonzero hyperideal $J$ of $R$.
\begin{Thm} \label{8832}
Let $R$ be a reduced hyperring. If an r-hyperideal $I$ of $R$ is not essential, then there exists a minimal prime hyperideal $P$ with $I \subseteq P$ such that $P$ is a maximal r-hyperideal. 
\end{Thm} 
\begin{proof}
Let $I$ be an r-hyperideal of $R$ that is not essential. Then there exists a nonzero hyperideal $J$ of $R$ such that $I \cap J \neq \{0\}$. Thus there exists a minimal prime hyperideal $P$ of $R$ such that $J \nsubseteq P$, since $R$ is a reduced hyperring. By Zorn$^,$s Lemma, we infer that there exists a maximal r-hyperideal $M$ with $I \subseteq M$ such that $J \cap M=\{0\}$. It means $MoJ = \{0\}$. Let $a \in J \backslash P$. Thus $Moa=\{0\}\subseteq P$. Since $P$ is a prime hyperideal, then $M \subseteq P$. By Theorem \ref{8831}, $M$ is a prime hyperideal. Thus $M=P$ and $I \subseteq P$. 
\end{proof}

\begin{Thm} \label{8833}
Let $I, I_1,...,I_n$ be hyperideals of $R$. Let $I \subseteq I_1 \cup ...\cup I_n$ such that no $I_i$ can be removed from the union. If one of the $I_i^,$s, say $I_1$, is an r-hyperideal and the others have regular elements, then $I \subseteq I_1$. 
\end{Thm}
\begin{proof}
Since no $I_i$ can be removed from the union
$I_1 \cup ...\cup I_n$, then there exist $x \in I$ but $x \notin I_2 \cup ...\cup I_n$. Hence $x \in I_1$. Let $y \in I \cap I_2 \cap...\cap I_n$. Therefore $x+y \notin I_2 \cup ... \cup I_n$. It is clear that $x+y \in I \subseteq I_1 \cup ...\cup I_n$. This implies that $y \in I_1$. It means $I \cap I_2 \cap ... \cap I_n \subseteq I_1$. Since $I_2 o ... o I_n \subseteq I_2 \cap ... \cap I_n$, then we have $IoI_2 o ... o I_n \subseteq I_1$. By Lemma \ref{21} (1), we have $I \subseteq I_1$, since $I_2 \cap ... \cap I_n \cap V(R) \neq \varnothing$.
\end{proof}
\begin{Cor} [Prime Avoidance Theorem] \label{8834}
Let $I, P_1,...,P_n$ be hyperideals of $R$. Let $I \subseteq P_1 \cup ...\cup P_n$ such that no $P_i$ can be removed from the union. If one of the $P_i^,$s, say $P_1$, is a minimal prime hyperideal of $R$ and the others have regular elements, then $I \subseteq P_1$. 
\end{Cor}
A nonempty subset $S$ of $R$ is called an r-multiplicatively closed subset of $R$ if $S$ contains at least a regular element $u \neq 1$, $0 \notin S$, $1 \in S$ and for all regular
elements $r \in S$ and all $a \in S$, $roa \subseteq S$ . 
\begin{Rem} \label{8835}
Let $S$ be an r-multiplicatively closed subset of $R$ and $T$ be multiplicatively closed subset of $R$ containing a regular element. Put $D=S \cup T\cup \{u \ \vert u \in sot, \ s \in S, t \in T\}$. Then $D$ is an r-multiplicatively closed subset of $R$. 
\end{Rem}
\begin{Cor} \label{8836}
$I$ is an r-hyperideal of $R$ if and only if $R - I$ is an r-multiplicatively closed subset of $R$. 
\end{Cor}
\begin{Thm} \label{8837}
Let $S$ be an r-multiplicatively closed subset of $R$ and $K$ be an hyperideal of $R$ disjoint from $S$. Then there exists a hyperideal $I$ which is maximal in the set of all hyperideals of $R$ disjoint from $S$, containing $K$. Any such hyperideal $I$ is an r-hyperideal of $R$.
\end{Thm}
\begin{proof}
Let $\Delta$ be the set of all hyperideals of $R$ disjoint from $S$, containing $K$. Then $\Delta \neq \varnothing$, since $K \in \Delta$ .So $\Delta$ is a partially ordered set with respect to set inclusion relation. By 
Zorn$^,$s lemma, there exists a hyperideal $I$ which is maximal in $\Delta$. Let $xoy \subseteq I$ for some $x,y \in R$ and $ann(x)=\{0\}$. Let $y \notin I$. Since $y \in (I:x)$ then $(I:x)$ properly contains $I$. Now let $u \in (I:x) \cap S$. Hence $uox \subseteq I$. By Remark \ref{8835}, we have $x \in V(R) \subseteq S$. Therefore $uox \subseteq S$ which means $uox \subseteq I \cap S$. This is conradictory to the fact that $I \cap S = \varnothing$. Hence $(I:x) \cap S= \varnothing$. Consequently, $y \in I$. Thus $I$ is an r-hyperideal of $R$.
\end{proof}
\section{$n$-hyperideals }
In this section, we introduce the notion of n-hyperideals in a commutative multiplicative hyperring and we
investigate many properties of them.
\begin{Def}
A hyperideal $I$ of a multiplicative hyperring $R$ is called an n-hyperideal of $R$, if for all $x,y \in R$, $xoy \subseteq I$ and $x \notin r(0)$, then $y \in I$.
\end{Def}
\begin{Thm} \label{23}
Every n-hyperideal of $R$ is an r-hyperideal.
\end{Thm}
\begin{proof}
Let $I$ be an n-hyperideal of $R$. Assume that $xoy \subseteq I$ with $ann(x)=\{0\}$ for some $x,y \in R$. Thus $x \notin r(0)$. Since $I$ is an n-hyperideal of $R$, we get $y \in I$ and therefore $I$ is an r-hyperideal. 
\end{proof}
\begin{Thm} \label{24}
Let $<0>$ be a primary hyperideal of $R$. Then the n-hyperideals and r-hyperideals are equivalent.
\end{Thm}
\begin{proof}
It is clear. 
\end{proof}
\begin{Thm} \label{21}
Let $I$ be an n-hyperideal of $R$. Then $I \subseteq r(0)$. 
\end{Thm}
\begin{proof} 
Let $I \nsubseteq r(0)$. It means that there exists an $x \in I$ but $x \notin r(0)$. Since $I$ is a {\bf C}-hyperideal and $x \in xo1$, then we have $xo1 \subseteq I$. Since
$I$ is an n-hyperideal of $R$, we have $1 \in I$ which means $I = R$. This is a contradiction. Thus $I \subseteq r(0)$. 
\end{proof}
\begin{Thm} \label{22}
The intersection of a nonempty set of n-hyperideals in $R$, is an n-hyperideal.
\end{Thm}
\begin{proof}
Let $\{I_i\}_{i \in \Omega}$ be a nonempty set of n-hyperideals in $R$ and $I=\bigcap_{i \in \Omega} I_i$. Assume that for $x,y \in R$, $xoy \subseteq I$ such that $x \notin r(0)$. Then we have $xoy \subseteq I_i$ for every $i \in \Omega$. Since $I_i$ is an n-hyperideal for every $i \in \Omega$ and $x \notin r(0)$, we obtain $y \in I_i$. It means $y \in I$.
\end{proof}

\begin{Thm} \label{25}
Let $I$ be a proper hyperideal of $R$. Then the following statements are equivalent:\\
(1) $I$ is an n-hyperideal of $R$.\\
(2) $I = (I : a)$ for every $a \notin r(0)$.\\
(3) $I_1oI_2 \subseteq I$ such that $I_1 \cap (R - r(0)) \neq \varnothing$, for hyperideals $I_1, I_2$ of $R$, implies $I_2 \subseteq I$. 
\end{Thm}
\begin{proof}
$(1) \Longrightarrow (2)$ Since $I \subseteq (I:x)$, we must show $(I:x) \subseteq I$. Let $y \in (I:x)$ and $x \notin r(0)$. Therefore $xoy \subseteq I$. Since $I$ is an n-hyperideal of $R$, we get $y \in I$. Consequently, $I=(I:x)$.\\
$(2) \Longrightarrow (3)$ Let $I_1oI_2 \subseteq I$ such that $I_1 \cap (R - r(0)) \neq \varnothing$. Thus there
exists an $x \in I_1$ but $x \notin r(0)$. Hence $xoI_2 \subseteq I$. It means that $I_2 \subseteq (I : x) = I$. \\
$(3) \Longrightarrow (1)$ Assume that $xoy \subseteq I$ but $x \notin r(0)$ for some $x,y \in R$. If we consider $I_1=<x>$ and $I_2=<y>$, then we are done.
\end{proof}
\begin{Cor}
Let $L$ be a hyperideal of R such that $L \cap (R-r(0)) \neq \varnothing$. If $I,J$ are n-hyperideals of $R$ such that $IoL=JoL$, then $I=J$. 
\end{Cor}
\begin{Thm} \label{26}
Let $I$ be a prime hyperideal of $R$. $I$ is an n-hyperideal of $R$ if and only if $I = r(0)$. 
\end{Thm}
\begin{proof}
$ \Longrightarrow$ Let $I$ be a prime hyperideal of $R$ such that $I$ is n-hyperideal. The inclusion $r(0) \subseteq I$ always holds. By Theorem \ref{21},
we obtain $I \subseteq r(0)$. Consequently, $ I=r(0)$. \\
$ \Longleftarrow $ Let $xoy \subseteq I$ such that $x \notin r(0)$ for some $x,y \in R$. Since $I$ is a prime ideal, we conclude that $y \in I$. It means $I$ is an n-hyperideal of
$R$. 
\end{proof} 
\begin{Cor} \label{27}
$r(0)$ is a prime hyperideal of $R$ if and only if it is an n-hyperideal of $R$.
\end{Cor}
\begin{proof}
$ \Longrightarrow$  follows from Theorem \ref{26}.\\
$ \Longleftarrow $ Assume that $xoy \subseteq r(0)$ such that $x \notin r(0)$. We get $y \in r(0)$, because $r(0)$ is an n-hyperideal of $R$. It means $I$ is a prime hyperideal.
\end{proof} 

\begin{Thm} \label{28}
Let $I$ be an n-hyperideal of $R$ and $T$ a nonempty subset of $R$ such that $T \nsubseteq I$, then $(I : T)$ is an
n-hyperideal of $R$. 
\end{Thm}
\begin{proof}
Let $xoy \subseteq (I : T)$ but $x \notin r(0)$. It means that $xoyot \subseteq I$ for every $t \in T$. Therefore for $a \in yot$, $xoa \subseteq xoyot \subseteq I$. Since
$I$ is an n-hyperideal of $R$, we get $a \in I$. Since $I$ is a {\bf C}-hyperideal and $yot \cap I \neq \varnothing$, then $yot \subseteq I$. So $y \in (I : T)$. 
\end{proof}
\begin{Thm} \label{29}
Let $I$ be a maximal n-hyperideal of $R$. Then $I=r(0)$. 
\end{Thm}
\begin{proof}
It is sufficient to show $I$ is a prime hyperideal of $R$. Then by Theorem \ref{26}, we are done. Let $xoy \subseteq I$ for some $x,y \in R$ such that $x \notin I$. By Theorem \ref{25}, $(I:x)$ is an n-hyperideal of $R$. Thus by maximality of $I$ we have $y \in (I:x)=I$. Consequently, $I$ is a prime hyperideal.
\end{proof}
\begin{Thm} \label{30}
$r(0)$ is a prime hyperideal of $R$ if and only if there exists an n-hyperideal of $R$.
\end{Thm}
\begin{proof}
$\Longrightarrow$ follows from Corollary \ref{27}.\\
$\Longleftarrow$
Let $\Delta$ be the set of all n-hyperideals of $R$ and $I$ an n-hyperideal of $R$. Then $\Delta \neq \varnothing$, since $I \in \Delta$. So $\Delta$ is a partially ordered set with respect to set inclusion relation. Now we take the chain $J_1 \subseteq J_2 \subseteq ... \subseteq J_n \subseteq ...$ in $\Delta$. Let $J=\bigcup_{i=1}^{\infty}J_i$ and let $xoy \subseteq J$ for some $x,y \in R$ such that $x \notin r(0)$. It means that there exists $t \in \mathbb{N}$ such that $xoy \subseteq J_t$. Since $J_t$ is an n-hyperideal, we get $y \in J_t \subseteq J$. Hence $J$ is an upper bound of the chain, then $\Delta$ has a maximal element $L$, by Zorn$^,$s Lemma. Thus $L=r(0)$, by Theorem \ref{29} and so it is prime hyperideal of $R$.
\end{proof} 
A nonempty subset $S$ of $R$ with $R -r(0) \subseteq S$ is called an n-multiplicatively closed subset of $R$ if for all $a \in R- r(0)$ and all $b \in S$, $aob \subseteq S$ . 
\begin{Thm}
Let $I$ be a hyperideal of the hyperring $R$. Then $I$ is an n-hyperideal of $R$ if and only if $R - I$ is an n-multiplicatively closed subset of $R$. 
\end{Thm}
\begin{proof}
$\Longrightarrow$ By Theorem \ref{21}, $I \subseteq r(0)$. Thus $R-r(0 ) \subseteq R- I$. Let $a \in R-r(0)$ and $b \in R-I$. If $aob \subseteq I$, then we get $b \in I$ which is a contradiction. Thus $aob \subseteq R-I$. It means that $R - I$ is an n-multiplicatively closed subset of $R$. \\
$\Longleftarrow$ Let $xoy \subseteq I$ for some $x,y \in R$ such that $x \notin r(0)$. If $y \notin I$, then $xoy \in R-I$. This is a contradiction. Hence $y \in I$ and then $I$ is an n-hyperideal of $R$. 
\end{proof}
\begin{Thm} \label{oooo}
Let $S$ be an n-multiplicatively closed subset of $R$ and $K$ be an hyperideal of $R$ disjoint from $S$. Then there exists a hyperideal $I$ which is maximal in the set of all hyperideals of $R$ disjoint from $S$, containing $K$. Any such hyperideal $I$ is an n-hyperideal of $R$.
\end{Thm}
\begin{proof}
Let $\Delta$ be the set of all hyperideals of $R$ disjoint from $S$, containing $K$. Then $\Delta \neq \varnothing$, since $K \in \Delta$. So $\Delta$ is a partially ordered set with respect to set inclusion relation. As in the proof of Theorem
\ref{30} we can show that $\Delta$ is bounded above. Hence by 
Zorn$^,$s Lemma, there exists a hyperideal $I$ which is maximal in $\Delta$. Let $xoy \subseteq I$ for some $x,y \in R$. If $x \notin r(0)$ and $y \notin I$, then we have $y \in (I:x)$ and $(I:x)$ properly contain $I$ . So, by maximality of $I$ in $\Delta$, we get $(I : x) \cap S \neq \varnothing $. It means that there exists an $s$ in $S$ with $s \in (I : x)$. Thus $sox \subseteq I$. Since $S$ is an n-multiplicatively closed subset of $R$ and $x \notin r(0)$, then we have $sox \subseteq S$. It implies that $S \cap I \neq \varnothing$ which is contradictory to the fact that $I \in \Delta$. Therefore $I$ is an n-hyperideal of $R$.
\end{proof}
\begin{Thm} 
Let $I \subseteq \bigcup _{i=1}^n I_i$ for some hyperideals $I_1,...,I_n$ of $R$. Suppose that there exists $1 \leq t \leq n$ such that $I_t$ is an n-hyperideal and others are without nilpotent elements. If $I \nsubseteq \bigcup_{i \neq t}I_i$, then $I \subseteq I_t$.
\end{Thm} 
\begin{proof}
Without destroying our assumption, we may assume that $t=1$. So $I \nsubseteq \bigcup_{i =2}^nI_i$. It means that there exists $a \in I$ such that $a \notin \bigcup_{i =2}^nI_i$. It implies that $a \in I$. Take $b \in I \cap (\bigcap_{i=2}^n)$. For every $2 \leq i \leq n$, we have $a \notin I_i$ and $b \in I_i$. It means $a+b \notin I_i$, for every $2 \leq i \leq n$. So $a+b \in I$ but $a+b \notin \bigcup_{i=2}^n I_i$. Thus we have $a+b \in I_1$ and then $b \in I_1$ since $a \in I_1$. It follows that $I \cap (\bigcap_{i=2}^nI_i) \subseteq I_1$. Therefore $(R-r(0)) \cap (\Pi_{i=2}^n) \neq \varnothing$, Since the product of non-nilpotent elements is non-nilpotent. Also, we have $Io(\Pi_{i=2}^n) \subseteq I_1$. Hence, by Theorem \ref{25}, we conclude that $I \subseteq I_1$. 
\end{proof}

\begin{Thm} \label{uuu}
(1) If $R$ is a reduced hyperring that is not integral hyperdomain, then $R$ has no n-hyperideal.\\
(2) Let $R$ be a reduced hyperring. Then 0 is an n-hyperideal of $R$ if and only if $R$ is an integral hyperdomain. 
\end{Thm} 
\begin{proof}
(1) By the assumption, $r(0) = 0$ is not a prime hyperideal of $R$. Thus $0$ is not an n-hyperideal, by Corollary \ref{27}. Also, let $I$ be a nonzero n-hyperideal of $R$. Then $I \subseteq r(0) = 0$ by Theorem \ref {21}. It means $I = 0$. This is a contradiction. \\
(2) $\Longrightarrow$ follows by (1).\\
$\Longleftarrow$ Let $R$ be an integral hyperdomain. Thus $r(0)=0$ is a prime hyperideal. Hence $0$ is an n-hyperideal of $R$ by Corollary \ref{27}.
\end{proof}
\begin{Thm} \label{ppp}
$0$ is the only n-hyperideal of $R$ if and only if $R$ is an integral hyperdomain.
\end{Thm}
\begin{proof}
$\Longrightarrow$ Let $0$ be the only n-hyperideal of $R$. Then $r(0)$ is prime hyperideal and n-hyperideal by Corollary \ref{27} and Theorem \ref{30}. Therefore $r(0) = 0$ is a prime hyperideal. Thus $R$ is an integral hyperdomain. \\
$\Longleftarrow$ Assume that $I$ is an n-hyperideal of $R$. Then $I \subseteq r(0) =0$ by Theorem \ref {21}. It means $I = 0$. 
\end{proof}

\begin{Ex} \label{exa} 
Let $(\mathbb{Z}, +,·)$ be the ring of integers. For all $x,y \in \mathbb{Z}$, we define $xoy=\{x \cdot a \cdot y: a \in A\}$ which $" \cdot "$ is ordinary multiplication and $A=\{5,7\}$. Then $(\mathbb{Z},+,o)$ is a multiplicative hyperring (see \cite{ref1}). We can show the hyperring by $\mathbb{Z}_A$. This hyperring is an integral hyperdomain. Thus $0$ is the only n-hyperideal of $\mathbb{Z}_A$.
\end{Ex}

\begin{Thm}
$R$ is an invertible if and only if $R$ is a regular hyperring and $0$ is an n-hyperideal. 
\end{Thm}
\begin{proof}
$\Longrightarrow$ Let $R$ be an invertible. Clearly, $R$ is a regular hyperring. Also, $0$ is an n-hyperideal by Theorem \ref{ppp}. \\
$\Longleftarrow$ Let $R$ be a regular hyperring and $0$ be an n-hyperideal of $R$. It is clear that $r(0)=0$. Let $0 \neq r \in R$. Since $R$ is a regular hyperring, then there exists $x \in R$
such that $r \in r^2ox$. Therefore $0 \in r^2ox-r$. Since $0$ is an n-hyperideal then $R$ is an integral hyperdomain, by Theorem \ref{uuu}. Then $0 \in r(rox -e)$ implies that $r=0$ which a contradiction or $0 \in rox -e$. This means $e \in rox$. Thus $R$ is an invertible.
\end{proof}

\section{Stability of $n$-hyperideals}
In this section, we investigate the stability of n-hyperideals in various hyperring-theoric constructions.

\begin{Thm} \label{32} 
Let $R_1$ and $R_2$ be multiplicative hyperrings and $ \phi:R_1 \longrightarrow R_2$ be a good homomorphism. Then the following statements hold : \\
(1) If $\phi $ is a monomorphism and $I_2$ is an n-hyperideal of $R_2$ , then $\phi^{-1} (I_2)$ is an n- hyperideal of $R_1$.\\
(2) If $\phi $ is an epimorphism and $I_1$ is an n-hyperideal of $R_1$ containing $Ker(\phi)$, then $\phi(I_1)$ is an n-hyperideal of $R_2$ .
\end{Thm}
\begin{proof}
(1) Let $xoy \subseteq \phi^{-1}(I_2)$ with $x,y \in R_1$ such that $x \notin r(0_{R_1})$. Thus we have $\phi(xoy)=\phi(x)o\phi(y) \subseteq I_2$. Since $\phi$ is a monomorphism , we conclude that $\phi(x)$ is not in $r(0_{R_2})$. It means that $\phi(y) \in I_2$, since $I_2$ is an n- hyperideal of $R_2$. Hence $y \in \phi^{-1}(I_2)$. Therefore $\phi^{-1}(I_2)$ is an n-hyperideal of $R_1$.\\
(2) Let $x_2oy_2 \subseteq \phi(I_1)$ with $x_2,y_2 \in R_2$ such that $x_2 \notin r(0_{R_2})$. Since $\phi$ is an epimorphism, there exist $x_1,y_1 \in R_1$ such that $\phi(x_1)=x_2, \phi(y_1)=y_2$ and so $\phi(x_1oy_1)=x_2oy_2 \subseteq \phi(I_1)$. Now take any $u \in x_1 o y_1 $. Then we get $\phi(u) \in \phi(x_1oy_1) \subseteq \phi(I_1)$ and so $\phi(u) = \phi(w)$ for some $w \in I_1$. This implies that $\phi(u-w) = 0 \in (0)$, that is, $u-w \in Ker(\phi) \subseteq I_1$ and so $u \in I_1$. Since $I_1$ is a ${\bf C}$-hyperideal of $R_1$,  we conclude that $x_1oy_1 \subseteq I_1$. Since $I_1$ is an n-hyperideal of $R_1$ and $x_1 \notin r(0_{R_1})$, then we have $y_1 \in I_1$. It means that $\phi(y_1)=y_2 \in \phi(I_1)$.Consequently, $\phi(I_1)$ is an n-hyperideal of $R_2$. 
\end{proof}

\begin{Cor}
Let $I, J$ be proper hyperideals of $R$ such that $J \subseteq I$ and $I$ be an n-hyperideal of $R$. Then $I / J$ is an n-hyperideal of $R / J$.
\end{Cor} 
\begin{proof}
Define $f:R \longrightarrow R/J$ by $f(r)=r+J$. Clearly, $f$ is a good epimorphism. Since $Ker (f)=J \subseteq I$ and $I$ is an n-hyperideal of $R$ , then the claim follows from Theorem \ref{32} (1).
\end{proof}
\begin{Cor}\label{33}
Let $I, J$ be proper hyperideals of $R$ such that $J \subseteq I$. If $I/J$ is an n-hyperideal of $R/J$ with $J \subseteq r(0)$, then $I$ is an n-hyperideal of $R$. 
\end{Cor}
\begin{proof}
Let $xoy \subseteq I$ for some $x,y \in R$ such that $x \notin r(0)$. It is clear that $(x+J)o(y+J)=\{z+J : z \in xoy \} \subseteq I/J$ and $x+J \notin r(0_{R/I})$. Thus $y+J \in I/J$, Since $I/J$ is an n-hyperideal of $R/J$. Therefore $y \in I$. It means that $I$ is an n-hyperideal of $R$. 
\end{proof}
\begin{Cor}
Let $I, J$ be proper hyperideals of $R$ such that $J \subseteq I$. If $J$ is an n-hyperideal of $R$ and $I/J$ is an n-hyperideal of $R/J$, then $I$ is an n-hyperideal of $R$. 
\end{Cor}
\begin{proof}
It is obvious by Theorem \ref{21} and Corollary \ref{33}.
\end{proof}
\begin{Thm} \label{11126} 
Let $R$ be a multiplicative hyperring with scalar
identity 1 and $I$ be a hyperideal of $R$. If $M_n(I)$ is an n-hyperideal of
$M_n(R)$, then $I$ is an n-hyperideal of $R$. 
\end{Thm}
\begin{proof}
Suppose that for $x,y \in R$ , $xoy \subseteq I$ such that $x \notin r(0_R)$. Then 
\[ \begin{pmatrix}
xoy & 0 & \cdots & 0\\
0 & 0 & \cdots & 0\\
\vdots& \vdots & \ddots \vdots\\
0 & 0 & \cdots & 0
\end{pmatrix}
\subseteq M_n(I) \]
It is clear that 
\[ \begin{pmatrix}
xoy & 0 & \cdots & 0\\
0 & 0 & \cdots & 0\\
\vdots& \vdots & \ddots \vdots\\
0 & 0 & \cdots & 0
\end{pmatrix}
=
\begin{pmatrix}
x & 0 & \cdots & 0\\
0 & 0 & \cdots & 0\\
\vdots& \vdots & \ddots \vdots\\
0 & 0 & \cdots & 0
\end{pmatrix}
\begin{pmatrix}
y & 0 & \cdots & 0\\
0 & 0 & \cdots & 0\\
\vdots& \vdots & \ddots \vdots\\
0 & 0 & \cdots & 0
\end{pmatrix}
.\]
Since $M_n(I)$ is an n-hyperideal of $M_n(R)$ and 
\[ \begin{pmatrix}
x & 0 & \cdots & 0\\
0 & 0 & \cdots & 0\\
\vdots& \vdots & \ddots \vdots\\
0 & 0 & \cdots & 0
\end{pmatrix} \notin r(0_{M_n(R)})\]
then we have

\[\begin{pmatrix}
y & 0 & \cdots & 0\\
0 & 0 & \cdots & 0\\
\vdots& \vdots & \ddots \vdots\\
0 & 0 & \cdots & 0
\end{pmatrix}
\in M_n(I).\]
It follows that $ y \in I$. Therefore $I$ is an n-hyperideal of $R$.

\end{proof}
\begin{Thm}
Let $T$ be a subhyperring of $R$. If I is an n-hyperideal of $R$ such that $T \nsubseteq I$, then $I \cap T$ is an n-hyperideal of $T$. 
\end{Thm}
\begin{proof}
Define $j:T \longrightarrow R$ by $j(t)=t$. It is clear that $j^{-1}(I)=I \cap T$. Thus $I \cap T$ is an n-hyperideal of $T$, by \ref{32} (1). 
\end{proof}

\begin{Thm}
Let $R_1$ and $R_2$ be two multiplicative hyperrings. If $I=I_1 \times I_2$ is an n-hyperideal of $R=R_1 \times R_2$, then $I=R$.
\end{Thm}
\begin{proof}
It is clear that $(1_{R_1},0_{R_2}) o(0_{R_1},1_{R_2}) \subseteq I$ but none of the elements are in $r(0_{R_1 \times R_2})$. Since $I$ is an n-hyperideal of $R$, then $(1_{R_1},0_{R_2})$ and $(0_{R_1},1_{R_2})$ are in $I$. Thus $I=R$.
\end{proof}
Let $(R, +, o)$ be a hyperring. We define the relation $\gamma$ on $R$ as follows:\\ 
$a \gamma b$ if and only if $\{a,b\} \subseteq U$ where $U$ is a finite sum of finite products of
elements of R, i.e.,$a \gamma b$ if and only if there exist $ z_1, ... , z_n \in R$ such that $\{a, b\} \subseteq \sum_{j \in J} \prod_{i \in I_j} z_i; \ \ I_j, J \subseteq \{1,... , n\}$

We denote the transitive closure of $\gamma$ by $\gamma ^{\ast}$. The relation $\gamma ^{\ast}$ is the smallest equivalence relation on a multiplicative hyperring $(R, +, o)$ such that the
quotient $R/\gamma ^{\ast}$, the set of all equivalence classes, is a fundamental ring. Let $\mathfrak{U}$
be the set of all finite sums of products of elements of R we can rewrite the
definition of $\gamma ^{\ast}$ on $R$ as follows:\\
$a\gamma ^{\ast}b$ if and only if there exist $ z_1, ... , z_n \in R$ with $z_1 = a, z_{n+1 }= b$ and $u_1, ... , u_n \in \mathfrak{U}$ such that
$\{z_i, z_{i+1}\} \subseteq u_i$ for $i \in \{1, ... , n\}$.
Suppose that $\gamma ^{\ast}(a)$ is the equivalence class containing $a \in R$. Then, both
the sum $\oplus$ and the product $\odot$ in $R/\gamma ^{\ast}$ are defined as follows:$\gamma^{\ast}(a) \oplus \gamma ^{\ast}(b)=\gamma ^{\ast}(c)$ for all $c \in \gamma^{\ast}(a) + \gamma ^{\ast}(b)$ and $\gamma ^{\ast}(a) \odot \gamma ^{\ast}(b)=\gamma ^{\ast}(d)$ for all $d \in  \gamma ^{\ast}(a) o \gamma ^{\ast}(b)$
Then $R /\gamma ^{\ast}$ is a ring, which is called a fundamental ring of $R$ (see also \cite{sorc4}).

\begin{Thm}
Let $R$ be a multiplicative hyperring with scalar identity $1$. Then the hyperideal $I$ of $R$ is n-hyperideal if and only if $I/\gamma ^{\ast}$ be an n-ideal of $R/\gamma ^{\ast}$. 
\end{Thm}
\begin{proof}
($\Longrightarrow$) Let for $x, y \in R/\gamma ^{\ast}, \ x \odot y \in I/\gamma ^{\ast}$ such that $x \notin r(0_{R/\gamma ^{\ast}})$. Thus, there exist $a, b\in R$ such
that $x =\gamma^{\ast}(a), y = \gamma^{\ast}(b)$ and $x \odot y = \gamma^{\ast}(a) \odot \gamma^{\ast}(b) =\gamma^{\ast}(a o b)$. So, $\gamma^{\ast} (a) \odot \gamma^{\ast}(b)=
\gamma^{\ast}(ao b) \in I/\gamma^{\ast}$, then $a ob \subseteq I$. Since $I$ is an n-hyperideal and $ a \notin r(0_R)$, then $b \in I$. Hencey $y= \gamma^{\ast}(b) \in I/\gamma ^{\ast}$.
Thus $I/\gamma ^{\ast}$ is an n-ideal of $R/\gamma ^{\ast}$.\\
($\Longleftarrow$) Suppose that $aob \subseteq I$ for $a, b \in R$ such that $a \notin r(0_R)$, then $\gamma^{\ast}(a), \gamma^{\ast}(b)\in R/\gamma^{\ast} $ and
$\gamma^{\ast}(a) \odot\gamma^{\ast}(b) = \gamma^{\ast}(a o b) \in I/\gamma^{\ast}$. Since $I/\gamma ^{\ast}$ is an n-ideal of $R/\gamma ^{\ast}$ and $\gamma^{\ast}(a) \notin r(0_{R/\gamma ^{\ast}})$, then we have $\gamma^{\ast}(b) \in I/\gamma^{\ast}$. It means that $b \in I$. Hence $I$ is an n-hyperideal of $R$.
\end{proof}

\end{document}